\documentclass[a4paper]{amsart}
\pdfoutput=1 %
\usepackage[utf8]{inputenc}
\usepackage[T1]{fontenc}
\usepackage{lmodern}

\usepackage[svgnames]{xcolor}

\usepackage[english]{babel}

\usepackage{csquotes}

\usepackage[backend=biber,
            sorting=nyt,
            sortcites=true,
            maxbibnames=9,
            maxalphanames=4,
            giveninits=true,
            url=false,
            isbn=false]{biblatex}
\DeclareFieldFormat{doi}{%
  \mkbibacro{DOI}\addcolon\space
  \ifhyperref
    {\href{https://doi.org/#1}{\MakeLowercase{\nolinkurl{#1}}}}
    {\MakeLowercase{\nolinkurl{#1}}}}
\DeclareSourcemap{
  \maps{
    \map{
      \pertype{online}
      \step[fieldset=url, null]
      \step[fieldset=urldate, null]
    }
  }
}
\DeclareSourcemap{
  \maps{
    \map{
      \pertype{book}
      \step[fieldset=pagetotal, null]
    }
  }
}
\usepackage{amssymb}
\usepackage[colorlinks,linkcolor=Maroon,citecolor=Teal,urlcolor=Navy]{hyperref}
\usepackage[capitalise]{cleveref}

\usepackage{xurl}
\hypersetup{breaklinks=true}

\usepackage[final]{microtype}

\newtheorem{theorem}{Theorem}

\newtheorem{proposition}[theorem]{Proposition}
\newtheorem{corollary}[theorem]{Corollary}
\theoremstyle{definition}
\newtheorem{definition}[theorem]{Definition}

\theoremstyle{remark}

\numberwithin{equation}{section}
\numberwithin{theorem}{section}

\newcommand{\N}{\mathbb{N}}
\newcommand{\R}{\mathbb{R}}

\newcommand{\abs}[1]{\left\lvert#1\right\rvert}

\newcommand{\dx}[1]{\mathop{d#1}}

\newcommand{\der}[5]{\frac{{#4}^{#3} #1}{#5 {#2}^{#3}}}
\newcommand{\pd}[2]{\der{#1}{#2}{}{\partial}{\partial}}
\newcommand{\pddm}[3]{\frac{\partial^2 #1}{\partial #2 \partial #3}}
\newcommand{\pdn}[3]{\der{#1}{#2}{#3}{\partial}{\partial}}
\newcommand{\pdd}[2]{\pdn{#1}{#2}{2}}
\newcommand{\td}[2]{\der{#1}{#2}{}{d}{d}}

\newcommand{\dd}[2]{\der{#1}{#2}{}{\delta}{\delta}}
\newcommand{\ddn}[3]{\der{#1}{#2}{#3}{\delta}{\delta}}
\newcommand{\mth}[2]{#1^{(#2)}}

\newcommand{\e}{\mathrm{e}}

\newcommand{\SO}{\mathrm{SO}}

\let\epsilon\varepsilon

\title[Singularities of low entropy curve shortening flow]{Singularities of low entropy high codimension curve shortening flow}
\author{Florian Litzinger}
\address{Institut für Analysis und Numerik, Otto-von-Guericke-Universität Magdeburg, Universitätsplatz 2, 39106 Magdeburg, Germany}
\email{florian.litzinger@ovgu.de}
\date{April 5, 2023}

\makeatletter
\hypersetup{pdftitle=\@title,pdfauthor=\authors}
\makeatother

\begin{document}
\begin{abstract}
  We consider curve shortening flow of arbitrary codimension in an Euclidean background.
  We show that, close to a singularity, the flow is asymptotically planar, paralleling Altschuler's work in the case of space curves, and analyse the blow-up limits of the flow.
  Using these results, we then prove that the curve shortening flow of initial curves with an entropy bound converges to a round point in finite time.
\end{abstract}

\maketitle

\section{Introduction}
\label{sec:Introduction}

Let $\gamma: S^1 \times [0, \omega) \to \R^n$ be a family of smooth immersions of the unit circle satisfying curve shortening flow with smooth initial data $\gamma_0: S^1 \to \R^n$:
\begin{gather}\label{eq:CSF}\tag{CSF}
  \begin{split}
    \pd{\gamma}{t}(p,t) &= (\kappa N)(p,t), \\
    \gamma(p, 0) &= \gamma_0(p).
  \end{split}
\end{gather}
Here, $\kappa(\cdot, t)$ is the curvature of $\gamma_t := \gamma(\cdot, t)$ and $N(\cdot, t)$ our choice of unit normal vector field. For brevity, we also say that $\{\gamma_t\}$ is a curve shortening flow. While in the planar case, $n=2$, there is a notion of the normal vector pointing inwards or outwards, in arbitrary codimension every curve $\gamma_t$ still has two well-defined orientations. We can thus choose the sign of the orientation to make \labelcref{eq:CSF} weakly parabolic forward in time. In particular, the product $\kappa N$ makes sense even at the points where $N$ is not defined. We also assume that every curve $\gamma_t$ is smooth and rectifiable.
It is well known that solutions exist at least for a short time \cite{gage1986heat}.

For curve shortening flow of embedded curves in the plane, Gage and Hamilton \cite{gage1984curve,gage1986heat} showed that convex curves eventually become circular and shrink to a point. In particular, the flow stays smooth throughout the evolution until the curvature tends to infinity at the final time. Grayson \cites{grayson1987heat,grayson1989shape}[see also][]{huisken1998distance} showed that any embedded curve continues to be embedded and eventually becomes convex without developing singularities in the process, thus completing the analysis of the flow of embedded curves in the plane. Immersed curves, on the other hand, can exhibit much more complex behaviour \cite{grayson1989shape}. The class of solutions that move self-similarly under the flow has been classified \cite{abresch1986normalized,halldorsson2012selfsimilar}.

Recently, both curve shortening flow as well as its higher-dimensional equivalent, the mean curvature flow, have been increasingly considered in codimensions greater than one.
For curves, this situation is considerably different from the planar case; for instance, an initially embedded solution may develop self-intersections over time. Moreover, common tools such as the maximum principle are not as applicable as in the codimension one case.
While Altschuler had studied singularities of space curves \cite{altschuler1991singularities} and Altschuler--Grayson defined a flow through singularities for planar curves by means of a special class of space curves \cite{altschuler1992shortening} in the early 1990s, more results concerning also the flow of curves immersed in arbitrarily dimensional spaces have begun to appear \cite{yang2005curve,ma2007curve,he2012distance,altschuler2013zoo,hattenschweiler2015curve,khan2015condition,corrales2016non,minarcik2020longterm,pan2021singularities}.
In mean curvature flow, Ambrosio and Soner \cite{ambrosio1997measuretheoretic} developed an approach in arbitrary codimension using a varifold ansatz, while the more traditional submanifold approach has been followed as well \cites{smoczyk2005selfshrinkers,baker2010mean,cooper2011mean}[see also][]{wang2008lectures,smoczyk2012mean}.
Notably, a notion of mean curvature flow with surgery in any codimension has been introduced by Nguyen \cite{nguyen2020high}.

More and more attention has also been paid to the study of flows with entropy bounds.
The entropy $\lambda(\gamma)$ of a curve $\gamma$, a functional that can be seen as a measure of geometric complexity, is defined by \cite{magni2009remarks,colding2012generic}
\begin{equation*}
  \lambda(\gamma) = \sup_{x_0 \in \R^n,\, t_0 > 0} (4 \pi t_0)^{- \frac{1}{2}} \int_\gamma \e^{- \frac{\abs{x - x_0}^2}{4 t_0}} \dx{\mu}\!.
\end{equation*}
Due to Huisken's monotonicity formula, the entropy is monotone non-increasing under curve shortening flow. Therefore, it is particularly appealing to consider in the study of singularities, since an upper bound on this quantity for the initial curve propagates with the flow.

In particular, entropy has been employed in the study of generic singularities of mean curvature flow by  Colding--Minicozzi \cite{colding2012generic} and Chodosh et al. \cite{chodosh2020mean,chodosh2021mean} and in Bernstein--L.~Wang's low-entropy Schoenflies theorem \cite{bernstein2022closed}. Colding et al. showed that the round sphere minimises entropy among closed self-shrinkers \cite{colding2013sphere}, while Bernstein--L.~Wang \cites{bernstein2016sharp,bernstein2017topological}[see also][]{bernstein2018topology} and Ketover--Zhou \cite{ketover2018entropy} proved the same statement for closed embedded surfaces in $\R^3$. An extension to higher dimensions is due to Zhu \cite{zhu2020entropy}. Hershkovits--White proved sharp entropy bounds for self-shrinkers of any dimension \cite{hershkovits2019sharp}. Bernstein--L.~Wang \cite{bernstein2018hausdorff} and S.~Wang \cite{wang2020spheres} proved the Hausdorff stability of round spheres under small perturbations of the entropy.
Moreover, in any codimension, Colding and Minicozzi \cite{colding2019entropy} gave uniform bounds on the entropy and codimension of generic singularities.

In this work, we first follow the strategy of Altschuler's work on space curves \cite{altschuler1991singularities} to show that singularity formation is an essentially planar phenomenon. That is, we show that blow-up limits of the flow are confined to two-dimensional subspaces of $\R^n$.
This result already appeared in the work of Yang and Jiao \cite{yang2005curve}.
Moreover, we argue that, as in the $n=3$ case, for any blow-up sequence of a type-I singularity (that is, the curvature does not grow faster than $(\omega -t)^\frac{1}{2}$) there exists a subsequence such that a rescaling of the curve along it converges to a planar self-similarly shrinking solution, while for a type-II singularity (that is, it is not of type-I), there exists an essential blow-up sequence such that a sequence of rescalings along it converges to the translating Grim Reaper solution. Since the entropy of the Grim Reaper is known \cite{guang2019volume}, we are thus able to rule out the occurrence of type-II singularities altogether, and combined with the classification of self-similarly shrinking curves in the plane \cites{abresch1986normalized}[see also][]{halldorsson2012selfsimilar}, we can show that for curve shortening flows with initially low entropy, the only possible singularity is the round circle:
\begin{theorem}[Main Theorem]
  \label{thm:CSFConvergence}
  Suppose that $\gamma: S^1 \times [0,\omega ) \to \R^n$ is a smooth solution of \labelcref{eq:CSF} with initial data $\gamma(\cdot, 0) = \gamma_0$ and assume that the entropy of $\gamma_0$ satisfies
  \begin{equation*}
    \lambda(\gamma_0) \leq 2.
  \end{equation*}
  Then $\omega $ is finite, and the rescaled flow converges to the round circle.
\end{theorem}

The paper is structured as follows. In \cref{sec:Notation}, we introduce our notation as well as the entropy functional.
The following two sections parallel Altschuler's work on singularities of space curves. In \cref{sec:CSF}, we derive the evolution equations for the Frenet--Serret frame and obtain Bernstein-type estimates for the derivatives of the tangent vector, \cref{thm:CurvatureEstimates}. In \cref{sec:BlowUpLimits}, we show in \cref{thm:BlowUpLimit,thm:BlowUpIsPlanar} that close to a singularity the solution is asymptotically planar, that is, a subsequential limit of a sequence of rescalings is a family of convex planar curves. Moreover, close to a type-I singularity, \cref{thm:ConvergenceTypeI} implies that a sequence of rescalings along a blow-up sequence converges to a planar self-similarly shrinking solution, while for a type-II singularity, we show the existence of an essential blow-up sequence converging to the Grim Reaper in \cref{thm:ConvergenceTypeII}.
Finally, in \cref{sec:ConvergenceAnalysis}, we prove our main result, which combines the previous results on singularities of the curve shortening flow with an entropy bound on the initial curve to show that the flow converges to a round point in finite time. This represents the first such convergence result for curve shortening flow in arbitrary codimension.

\subsection*{Acknowledgement}
This work is part of the author's Ph.\,D. thesis at Queen Mary University of London.
Many thanks are due to Huy The Nguyen for suggesting the problem and his support throughout the project.

\section{Notation and Preliminaries}
\label{sec:Notation}

We consider curve shortening flow of curves in $\R^n$. It is convenient to work in the Frenet--Serret frame, and we will largely adopt the notation from Gage--Hamilton's work on planar curves \cite{gage1986heat}.

\subsection{Curves in \texorpdfstring{$\R^n$}{Rn}}

Let $\gamma: S^1 \to \R^n$, $p \mapsto \gamma(p)$ be an immersion of the unit circle into $\R^n$ with the parameter $p$ taken to be modulo $2 \pi$. We always assume that $\gamma$ is smooth and rectifiable. Define the velocity of the parametrisation by
\begin{equation*}
  v = \abs{\pd{\gamma}{p}}.
\end{equation*}
Since $\gamma$ is rectifiable, we can parametrise it by arclength $s$, so that $v(s) = 1$. 
Moreover, the differential operator $\pd{}{s}$ is given by
\begin{equation}\label{eq:dds}
  \pd{}{s} = \frac{1}{v(p)} \pd{}{p},
\end{equation}
and for any $U \subset S^1$, the induced measure $\dx{s}$ is given by
\begin{equation*}
  \int_U \dx{s} = \int_U v(p) \dx{p}\!.
\end{equation*}

Furthermore, let $(T,N,B_1,\dots,B_{n-2})$ denote the Frenet--Serret moving frame of orthonormal vectors, where $T$ is called the tangent vector, $N$ the normal vector and $B_1,\dots,B_{n-2}$ the binormal vectors.
The curvature $\kappa$ is given by
\begin{equation*}
  \kappa = \abs{\pd{T}{s}} = \abs{\pdd{\gamma}{s}},
\end{equation*}
and we let $\tau_1 = \abs{\pd{N}{s} + \kappa T}, \tau_2, \dots, \tau_{n-2}$ denote the torsions. Then the generalised Frenet--Serret equations hold \cite{spivak1999comprehensive4}:
\begin{align}\label{eq:FrenetSerret}
  \pd{}{s} \begin{pmatrix} T \\ N \\ B_1 \\ B_2 \\ \vdots \\ B_{n-2} \end{pmatrix} &= \begin{pmatrix}
    0 & \kappa & 0 & 0 & \dots & 0 \\
    -\kappa & 0 & \tau_1 & 0 & \dots & 0 \\
    0 & -\tau_1 &  0 & \tau_2 & \dots & 0 \\
    0 & 0 & -\tau_2 & 0 & \dots & 0 \\
    \vdots & \vdots & \vdots & \vdots & \ddots & \tau_{n-2} \\
    0 & 0 & 0 & 0 & -\tau_{n-2} & 0
  \end{pmatrix} \begin{pmatrix} T \\ N \\ B_1 \\ B_2 \\ \vdots \\ B_{n-2} \end{pmatrix}.
\end{align}

In the following sections, we will consider time-dependent immersions $\gamma_t = \gamma(\cdot, t)$ moving in the direction $N$ with speed $\kappa$ and compute the evolution equations for the Frenet--Serret frame.

Finally, in order to show that a blow-up limit of a solution of curve shortening flow is planar, we require the following theorem. It generalises the well-known fact that a space curve with vanishing torsion is planar to the case of curves in $\R^n$.
\begin{theorem}[Spivak {\cite[Thm.~7.B.5]{spivak1999comprehensive4}}]
  \label{thm:SpivakThm}
  Let $\gamma: S^1 \to \R^n$ be a curve parametrised by arclength such that $\kappa, \tau_1, \dots, \tau_{j-2}$ do not vanish at any point and $\tau_{j-1}$ vanishes everywhere. Then $\gamma$ lies in some $(j-1)$-dimensional plane in $\R^n$.
\end{theorem}

\subsection{Entropy}

We now define the entropy functional, which was introduced by Magni--Mantegazza \cite{magni2009remarks} and Colding--Minicozzi \cite{colding2012generic}, and state its basic properties.

Let $M^m \subset \R^n$ be an immersed surface.
For $(x_0, t_0) \in \R^n \times \R$ the scaled backward heat kernel centered at $(x_0, t_0)$ is given by
\begin{equation*}
  k_{x_0, t_0}(x, t) = (4 \pi (t_0 - t))^{- \frac{m}{2}} \e^{- \frac{\abs{x-x_0}^2}{4(t_0-t)}}.
\end{equation*}
Note that $k_{x_0,t_0}$ is well defined on $\R^n \times (-\infty, t_0)$.
Then Huisken's $F$-functional $F_{x_0, t_0}$, $x_0 \in \R^n$, $t_0 > 0$, is defined by \cite{huisken1990asymptotic,colding2012generic}
\begin{equation*}
  F_{x_0, t_0}(M) = (4 \pi t_0)^{- \frac{m}{2}} \int_M \e^{- \frac{\abs{x - x_0}^2}{4 t_0}} \dx{\mu} = \int_M k_{x_0, t_0}(x, 0) \dx{\mu}\!.
\end{equation*}
\begin{definition}
  The entropy of $M$ is defined by \cite{magni2009remarks,colding2012generic}
  \begin{equation*}
    \lambda(M) = \sup_{x_0 \in \R^n,\, t_0 > 0} F_{x_0, t_0}(M).
  \end{equation*}
\end{definition}
Suppose that $\{\gamma_t\}$ is a curve shortening flow of closed curves. For any $s < t < t_0$, Huisken's monotonicity formula \cite{huisken1990asymptotic} yields
\begin{equation*}
  \td{}{t} \int_{\gamma_t} k_{x_0, t_0} \dx{\mu} \leq 0,
\end{equation*}
whereby
\begin{equation*}
  F_{x_0, t_0} (\gamma_t) \leq F_{x_0, t_0 + (t-s)} (\gamma_s),
\end{equation*}
which implies that $\lambda(\gamma_t) \leq \lambda(\gamma_s)$ for any $s < t$. Therefore, if $\{\gamma_t\}$ is a curve shortening flow, $\lambda(\gamma_t)$ is non-increasing in $t$.
Moreover, the entropy functional $\lambda$ is non-negative and invariant under dilations, rotations and translations of $\gamma$, and the critical points of $\lambda$ are self-similarly shrinking solutions of the curve shortening flow, as shown by Colding--Minicozzi \cite{colding2012generic}.

In addition, among all closed planar curves, the entropy $\lambda$ is minimised by the circle. By the Gage--Hamilton--Grayson theorem, any closed curve $\gamma$ in the plane becomes convex and eventually shrinks to a round point. But the entropy is non-increasing under curve shortening flow, so we must have $\lambda(\gamma) \geq \lambda(S^1)$.

Note that the entropy of a self-shrinker is equal to the functional $F_{0,1}$, the Gaussian area, since by the monotonicity formula, the critical points of $F_{0,1}$ are precisely the self-shrinkers \cite{colding2012generic}.

For some examples, it is possible to compute the entropy explicitly. First of all, the entropy is normalised so that the entropy of a straight line is equal to one. Also, for any curve $\gamma$ we have $\lambda(\gamma \times \R) = \lambda(\gamma)$. For the circle $S^1$, Stone \cite{stone1994density} showed that
\begin{equation*}
  \lambda(S^1) = \sqrt{\frac{2\pi}{\e}} \approx 1.52.
\end{equation*}

We will also need the value of the entropy of some special solutions of the curve shortening flow.
The entropy of the translating Grim Reaper solution has been computed by Guang:
\begin{proposition}[Guang {\cite[Thm.~1.3]{guang2019volume}}]
  \label{prop:EntropyGR}
  Let $\Gamma: (-\frac{\pi}{2},\frac{\pi}{2}) \times (0, \infty) \to \R^2$, $\Gamma(x,t) = (x, - \log (\cos x) + t)$ denote the Grim Reaper. For any point $(x_0,y_0) \in \R^2$, $t_0 \in(0, \infty)$, we have that
  \begin{equation*}%
    F_{(x_0,y_0),t_0} (\Gamma) \leq 2.
  \end{equation*}
  In fact,
  \begin{equation*}%
    \lim_{N \to \infty} F_{(0,N),N} (\Gamma) = 2.
  \end{equation*}
  Therefore, the entropy of the Grim Reaper satisfies
  \begin{equation*}
    \lambda(\Gamma) = 2.
  \end{equation*}
\end{proposition}

Moreover, for the self-shrinking Abresch--Langer solutions \cite{abresch1986normalized} there is a lower bound on the entropy due to Baldauf and Sun.
\begin{proposition}[Baldauf--Sun {\cite{baldauf2020sharp}}]
  \label{prop:EntropyAL}
  Let $\gamma_{m,n}$ denote an Abresch--Langer curve, that is, a closed convex self-shrinking solution to the curve shortening flow in the plane with turning number $m$ and $2n$ critical points of the curvature, where $m,n \in \N$ are coprime integers such that
  \begin{equation*}
    \frac{1}{2} < \frac{m}{n} < \frac{\sqrt{2}}{2}.
  \end{equation*}
  Then the entropy of $\gamma_{m,n}$ satisfies
  \begin{equation*}
    \lambda(\gamma_{m,n}) \geq m \lambda(S^1) = m \sqrt{\frac{2 \pi}{\e}}.
  \end{equation*}
\end{proposition}

Finally, we have an estimate for the entropy in terms of the Euclidean density at each point.

\begin{proposition}[White {\cite{white2015topics}}]
  \label{prop:EntropyEuclideanDensity}
  Let $M^m \subset \R^n$ be an immersed surface. Define the Euclidean density of $M$ at $x$ by
  \begin{equation*}
    \Theta^m (M, x) = \lim_{r \to 0} \frac{\mathcal{H}^m (B_r(x) \cap M)}{\omega_m r^m},
  \end{equation*}
  where $\omega_m$ is the volume of the $m$-dimensional unit ball. Then, for any $x \in M$, we have that
  \begin{equation*}
    \lambda(M) \geq \Theta^m (M, x)
  \end{equation*}
  whenever the limit exists.
\end{proposition}

\section{Curve shortening flow in any codimension}
\label{sec:CSF}

Let $\gamma: S^1 \times [0, \omega ) \to \R^n$ be a one-parameter family of smooth immersions of curves satisfying \labelcref{eq:CSF}.
Then, for any smooth initial curve $\gamma_0: S^1 \to \R^n$, there exists a unique smooth solution on some time interval $[0, \omega )$, $0 < \omega \leq \infty$ \cite{gage1986heat}. In fact, by \cref{prop:EvolF} below, we have $\omega \leq \frac{1}{2} \max \abs{\gamma_0}^2$.

\subsection{Evolution equations}
\label{subsec:EvolEqns}

In the following propositions, we state the relevant evolution equations in terms of the Frenet--Serret frame. Their analogues have previously been derived for planar curves by Gage and Hamilton \cite{gage1986heat} and for space curves by Altschuler \cite{altschuler1991singularities}; in the general case some of them appeared in various places in the literature \cite{yang2005curve,ma2007curve,hattenschweiler2015curve}.

Recall that we have set $s$ to be the arclength parameter and $v = \abs{\pd{\gamma}{p}}$.
The following two statements can be proved exactly as in the planar case \cite{gage1986heat}.

\begin{proposition}
  The evolution of $v$ is given by
  \begin{equation*}
    \pd{}{t} v = -\kappa^2 v.
  \end{equation*}
\end{proposition}
\begin{proof}
  The operators $\pd{}{p}$ and $\pd{}{t}$ commute, hence
  \begin{align*}
    2 v \pd{}{t} v = \pd{}{t} (v^2) &= 2 \left\langle \pd{\gamma}{p}, \pddm{\gamma}{p}{t} \right\rangle \\
    &= 2 \left\langle v T, \pd{}{p} \left( \kappa N \right) \right\rangle \\
    &= 2 \left\langle v T, \pd{\kappa}{p} N - v \kappa^2 T + v \kappa \tau_1 B_1 \right\rangle \\
    &= -2 v^2 \kappa^2,
  \end{align*}
  where we have used the Frenet--Serret equations \labelcref{eq:FrenetSerret}, \labelcref{eq:CSF}, and that the vectors $(T, N, B_1)$ are orthonormal.
\end{proof}

Since the arclength parameter $s$ depends on $t$, we cannot expect that the operators $\pd{}{s}$ and $\pd{}{t}$ commute. Instead, we have

\begin{proposition}\label{prop:CommutationFormula}
  Differentiation with respect to $s$ and $t$ is related by the commutation formula
  \begin{equation}\label{eq:CommutationFormula}
  	\pddm{}{t}{s} = \pddm{}{s}{t}  + \kappa^2 \pd{}{s}.
  \end{equation}
\end{proposition}
\begin{proof}
  Using \cref{eq:dds}, we obtain
  \begin{equation*}
    \pd{}{t} \pd{}{s} = \kappa^2 \frac{1}{v} \pd{}{p} + \frac{1}{v} \pd{}{p} \pd{}{t} = \kappa^2 \pd{}{s} + \pd{}{s} \pd{}{t},
  \end{equation*}
  proving the claim.
\end{proof}

These two propositions and the Frenet--Serret equations \labelcref{eq:FrenetSerret} enable us to derive the evolution equations for the moving frame and the curvature and torsions.

\begin{proposition}
  We have
  \begin{equation}\label{eq:EvolT}
  	\pd{T}{t} = \pdd{T}{s} + \kappa^2 T = \pd{\kappa}{s} N + \kappa \tau_1 B_1.
  \end{equation}
\end{proposition}
\begin{proof}
  Using the commutation formula first, followed by the curve shortening flow equation \labelcref{eq:CSF}, we compute
  \begin{align*}
  	\pd{T}{t} &= \pddm{\gamma}{t}{s} \\
  	&= \pddm{\gamma}{s}{t} + \kappa^2 \pd{\gamma}{s} \\
  	&= \pdn{\gamma}{s}{3} + \kappa^2 T \\
  	&= \pdd{T}{s} +  \kappa^2 T.
  \end{align*}
  In order to obtain the second equality, observe that
  \begin{align}
  	\nonumber \pdd{T}{s} &= \left( \pd{\kappa}{s} \right) N + \kappa \pd{N}{s} \\
  	\label{eq:Tss} &= \pd{\kappa}{s} N + \kappa \left( -\kappa T + \tau_1 B_1 \right)
  \end{align}
  by \cref{eq:FrenetSerret}, so that
  \begin{equation*}
   	\pd{T}{t} =  \pd{\kappa}{s} N + \kappa \tau_1 B_1,
  \end{equation*}
  as required.
\end{proof}

\begin{proposition}
  The evolution of the curvature $\kappa$ is given by
  \begin{equation}\label{eq:EvolKappa}
  	\pd{\kappa}{t} = \pdd{\kappa}{s}  + \kappa^3 - \kappa \tau_1^2
  \end{equation}
  whenever $\kappa > 0$.
\end{proposition}
\begin{proof}
  We use the fact that $\kappa^2 = \langle \pd{T}{s}, \pd{T}{s} \rangle$ and get, using the Frenet--Serret equations \labelcref{eq:FrenetSerret} and \cref{eq:EvolT}, that
  \begin{align*}
  	\kappa \pd{\kappa}{t} &= \left\langle \pddm{T}{t}{s}, \pd{T}{s} \right\rangle \\
  	&= \left\langle \pddm{T}{s}{t} + \kappa^3 N, \kappa N \right\rangle \\
  	&= \left\langle \pdd{\kappa}{s} N + \pd{\kappa}{s} \pd{N}{s} + \pd{\kappa}{s} \tau_1 B_1 + \kappa \pd{\tau_1}{s} B_1 + \kappa \tau_1 \pd{B_1}{s} + \kappa^3 N, \kappa N \right\rangle \\
  	&= \kappa \pdd{\kappa}{s} + \kappa^4 + \langle \kappa^2 \tau_1 (-\tau_1 N + \tau_2 B_2), N \rangle \\
  	&= \kappa \pdd{\kappa}{s} + \kappa^4 -\kappa^2 \tau_1^2,
  \end{align*}
  which implies the claim.
\end{proof}

\begin{corollary}\label{cor:EvolKSquared}
  We have
  \begin{equation*}
  	\pd{\kappa^2}{t} = \pdd{\kappa^2}{s}  - 2 \left( \pd{\kappa}{s} \right)^2 + 2\kappa^4 -  2 \kappa^2 \tau_1^2.
  \end{equation*}
\end{corollary}

\begin{proposition}
  The normal vector field $N$ satisfies
  \begin{equation}\label{eq:EvolN}
  	\pd{N}{t} = - \pd{\kappa}{s} T + \left( \pd{\tau_1}{s} + \frac{2 \tau_1}{\kappa} \pd{\kappa}{s} \right) B_1 + \tau_1 \tau_2 B_2.
  \end{equation}
\end{proposition}
\begin{proof}
  From the Frenet--Serret equations \labelcref{eq:FrenetSerret} we have $\pd{T}{s} = \kappa N$, so
  \begin{equation*}
  	\pd{N}{t} = \pd{}{t} \left( \frac{1}{\kappa} \pd{T}{s} \right).
  \end{equation*}
  The result then follows by a computation similar to the ones above using the Frenet--Serret equations \labelcref{eq:FrenetSerret}, the commutation formula \labelcref{eq:CommutationFormula}, and the evolution equations for $T$ and $\kappa$, \cref{eq:EvolT,eq:EvolKappa}.
\end{proof}

\begin{proposition}
  The evolution of the first torsion $\tau_1$ is given by
  \begin{equation*}
  	\pd{\tau_1}{t} = \pdd{\tau_1}{s} + \frac{2}{\kappa} \pd{\kappa}{s} \pd{\tau_1}{s} + \frac{2 \tau_1}{\kappa} \left( \pdd{\kappa}{s} - \frac{1}{\kappa} \left( \pd{\kappa}{s} \right)^2 + \kappa^3 \right) - \tau_1 \tau_2^2.
  \end{equation*}
\end{proposition}
\begin{proof}
  Note that, by the Frenet--Serret equations \labelcref{eq:FrenetSerret},
  \begin{equation*}
    \left\langle \pd{N}{s}, \pd{N}{s} \right\rangle = \kappa^2 + \tau_1^2.
  \end{equation*}
  Differentiating this equation with respect to $t$ yields \cite{yang2005curve}
  \begin{equation*}
    \tau_1 \pd{\tau_1}{t} + \kappa \pd{\kappa}{t} = \left\langle \pddm{N}{t}{s}, \pd{N}{s} \right\rangle.
  \end{equation*}
  Using \cref{eq:FrenetSerret}, the commutation formula \labelcref{eq:CommutationFormula} and the evolution equations for $N$ and $\kappa$, \cref{eq:EvolN,eq:EvolKappa}, we obtain the result.
\end{proof}

\begin{proposition}\label{prop:EvolF}
  We have
  \begin{equation*}%
  	\pd{\abs{\gamma}^2}{t} = \pdd{\abs{\gamma}^2}{s} - 2.
  \end{equation*}
\end{proposition}
\begin{proof}
  Since $\pd{T}{s} = \kappa N$, the curve shortening flow equation can also be written as
  \begin{equation*}
    \pd{\gamma}{t} = \pdd{\gamma}{s}.
  \end{equation*}
  While this resembles a heat equation, note that the arclength parameter $s$ depends on the time $t$. This implies that
  \begin{align*}
  	\pd{\abs{\gamma}^2}{t} &= 2 \left\langle \pd{\gamma}{t}, \gamma \right\rangle \\
  	&= 2 \left\langle \pdd{\gamma}{s}, \gamma \right\rangle.
  \end{align*}
  Finally, note that
  \begin{align*}
  	\pdd{\abs{\gamma}^2}{s} &= \pdd{}{s} \left\langle \gamma, \gamma \right\rangle \\
  	&= 2 \left\langle \pdd{\gamma}{s}, \gamma \right\rangle + 2 \left\langle \pd{\gamma}{s}, \pd{\gamma}{s} \right\rangle \\
  	&= 2 \left\langle \pdd{\gamma}{s}, \gamma \right\rangle + 2,
  \end{align*}
  which proves the claim.
\end{proof}

\subsection{Scaling-invariant estimates}

As in the space curve case, we have scaling-invariant estimates on the derivatives of the tangent vector, and thus the derivatives of the curvature, depending only on the maximal curvature at the initial time.

We saw in \cref{eq:EvolKappa} that even in higher codimension, the evolution of the curvature still only involves the curvature $\kappa$ itself and the first torsion $\tau_1$, but none of the higher torsions. Hence we may argue exactly as Altschuler \cite{altschuler1991singularities} and obtain the following result.
For brevity, we write, e.\,g., $\pd{T}{s} \equiv T_s$ and $\mth{T}{m} \equiv \pdn{T}{s}{m}$ for the derivatives of the tangent vector $T$ \cite[cf.][Ch.~2.6]{andrews2020extrinsic}.

\begin{theorem}[cf. {\cites[Thm.~3.1]{altschuler1991singularities}[Thm.~3.1]{yang2005curve}[Thm.~3.6]{hattenschweiler2015curve}}]
  \label{thm:CurvatureEstimates}
  For any $m \geq 1$ there exists $C_m < \infty$ such that for $t \in (0, \frac{1}{8 K_0}]$, where $K_t := \sup \kappa^2(\cdot, t)$, it holds that
  \begin{equation*}
    \abs{\pdn{T}{s}{m}}^2 \leq \frac{C_m K_0}{t^{m-1}}.
  \end{equation*}
\end{theorem}

\begin{proof}
  The commutation formula \labelcref{eq:CommutationFormula} implies
  \begin{equation*}
    T_t = T_{ss} + \abs{T_s}^2 T,
  \end{equation*}
  because $\gamma$ evolves by curve shortening flow. Therefore,
  \begin{align*}
    \abs{T_s}^2_t &= 2 \langle (T_s)_t, T_s \rangle \\
    &= 2 \langle (T_t)_s + \abs{T_s}^2 T_s, T_s \rangle \\
    &= 2 \langle (T_{ss} + \abs{T_s}^2 T)_s, T_s \rangle + 2 \abs{T_s}^4 \\
    &= 2 \langle T_{sss}, T_s \rangle + 2 \abs{T_s}^2_s \langle T, T_s \rangle + 4 \abs{T_s}^4 \\
    &= \abs{T_s}^2_{ss} - 2 \abs{T_{ss}}^2 + 4 \abs{T_s}^4,
  \end{align*}
  as $\abs{T_s}^2_{ss} = 2 \langle T_{sss}, T_s \rangle + 2 \abs{T_{ss}}^2$. We thus have the differential inequality
  \begin{equation*}
    \abs{T_s}^2_t - \abs{T_s}^2_{ss} = - 2 \abs{T_{ss}}^2 + 4 \abs{T_s}^4 \leq 4 \abs{T_s}^4.
  \end{equation*}
  By the ODE comparison principle and $\abs{T_s}^2 = \kappa^2 \leq K_0$ at $t=0$, we have
  \begin{equation*}
    \abs{T_s}^2 \leq \frac{K_0}{1 - 4 K_0 t} \leq 2 K_0,
  \end{equation*}
  as $t \leq \frac{1}{8 K_0}$ by assumption. We thus choose $C_1 = 2$.

  We define
  \begin{equation*}
    Z_m := \mth{T}{m}_t - \mth{T}{m}_{ss},
  \end{equation*}
  using the notation $\mth{T}{m} = \pdn{T}{s}{m}$. Then
  \begin{equation*}
    Z_{m+1} = (Z_m)_s + \abs{T_s}^2 \mth{T}{m+1}
  \end{equation*}
  and
  \begin{align*}
    \abs{\mth{T}{m}}^2_t - \abs{\mth{T}{m}}^2_{ss} &= 2 \langle \mth{T}{m}_t - \mth{T}{m}_{ss}, \mth{T}{m} \rangle - 2 \abs{\mth{T}{m+1}}^2 \\
    &= 2 \langle Z_m, \mth{T}{m} \rangle - 2 \abs{\mth{T}{m+1}}^2.
  \end{align*}
  We already know
  \begin{align*}
    Z_0 &= \abs{T_s}^2 T, \\
    Z_1 &= 2 \langle T_{ss}, T_s \rangle T + 2 \abs{T_s}^2 T_s,
  \end{align*}
  and moreover
  \begin{equation*}
    Z_2 = 2 \langle T_{sss}, T_s \rangle T + 2 \abs{T_{ss}}^2 T + 6 \langle T_{ss}, T_s \rangle T_s + 3 \abs{T_{s}}^2 T_{ss}.
  \end{equation*}

  For $m=2$, define $\Phi := t \abs{T_{ss}}^2 + 4 \abs{T_s}^2$. Then
  \begin{align*}
    \Phi_t - \Phi_{ss}   %
    = &~ \abs{T_{ss}}^2 + 2t \langle (T_{ss})_t, T_{ss} \rangle + 4(\abs{T_s}^2_{ss} - 2 \abs{T_{ss}}^2 + 4 \abs{T_s}^4) \\
    \nonumber &- (2t \langle (T_{ss})_{ss}, T_{ss} \rangle + 2t \abs{T_{sss}}^2 + 4 \abs{T_s}^2_{ss}) \\
    = &~ -7 \abs{T_{ss}}^2 + 2t(\langle (T_{ss})_t - (T_{ss})_{ss}, T_{ss} \rangle - \abs{T_{sss}}^2) + 16 \abs{T_s}^4 \\
    = &~ -7 \abs{T_{ss}}^2 + 2t(\langle Z_2, T_{ss} \rangle - \abs{T_{sss}}^2) + 16 \abs{T_s}^4
  \end{align*}
  Since
  \begin{align*}
    \langle Z_2, T_{ss} \rangle = &~ 2 \langle T_{sss}, T_s \rangle \langle T, T_{ss} \rangle + 2 \abs{T_{ss}}^2 \langle T, T_{ss} \rangle \\
    \nonumber &+ 6 \langle T_{ss}, T_s \rangle^2 + 3 \abs{T_s}^2 \abs{T_{ss}}^2 \\
    = &~ 2 \langle T_{sss}, T_s \rangle \langle T, T_{ss} \rangle - 2 \abs{T_{ss}}^2 \abs{T_s}^2 \\
    \nonumber &+ 6 \langle T_{ss}, T_s \rangle^2 + 3 \abs{T_s}^2 \abs{T_{ss}}^2,
  \end{align*}
  where we have used that $T_{ss} = - \kappa^2 T + \kappa_s N + \kappa \tau_1 B_1$ (see \cref{eq:Tss}) and $\kappa = \abs{T_s}^2$ imply that $\langle T, T_{ss} \rangle = - \abs{T_s}^2$, we obtain
  \begin{align*}
    \Phi_t - \Phi_{ss} = &~ -7 \abs{T_{ss}}^2 + 4t \langle T_{sss}, T_s \rangle \langle T, T_{ss} \rangle - 4t \abs{T_{ss}}^2 \abs{T_s}^2 \\
    \nonumber &+ 12t \langle T_{ss}, T_s \rangle^2 + 6t \abs{T_s}^2 \abs{T_{ss}}^2 - 2t \abs{T_{sss}}^2 + 16 \abs{T_s}^4 \\
    \leq &~ -7 \abs{T_{ss}}^2 + 4t \abs{T_{sss}} \abs{T_{ss}} \abs{T_s} + 14t \abs{T_{ss}}^2 \abs{T_s}^2 - 2t \abs{T_{sss}}^2 + 16 \abs{T_s}^4 \\
    = &~ -7 \abs{T_{ss}}^2 - 2t(\abs{T_{sss}} - \abs{T_{ss}} \abs{T_s})^2 + 16t \abs{T_{ss}}^2 \abs{T_s}^2 + 16 \abs{T_s}^4 \\
    \leq &~ 64 K_0^2 + (32 K_0 t - 7) \abs{T_{ss}}^2 \\
    \leq &~ 64 K_0^2,
  \end{align*}
  using $\abs{T_s}^2 \leq 2 K_0$ and $t \leq \frac{1}{8 K_0}$. At $t=0$ we have $\Phi \leq 4 K_0$, so the ODE comparison principle implies
  \begin{equation*}
    \Phi \leq 64 K_0^2 t + 4 K_0 \leq 12 K_0
  \end{equation*}
  for any $t \leq \frac{1}{8 K_0}$. Therefore,
  \begin{equation*}
    \abs{T_{ss}}^2 \leq \frac{12 K_0}{t},
  \end{equation*}
  and we may choose $C_2 = 12$.

  In general, for $m \geq 3$ we have
  \begin{align}%
    \label{eq:mthformula}
    Z_m = &~ 2 \langle \mth{T}{m+1}, \mth{T}{1} \rangle T \\
    \nonumber &+ 2m \langle \mth{T}{m}, \mth{T}{2} \rangle T \\
    \nonumber &+ 2(m+1) \langle \mth{T}{m}, \mth{T}{1} \rangle \mth{T}{1} \\
    \nonumber &+ (m+1) \abs{\mth{T}{1}}^2 \mth{T}{m} \\
    \nonumber &+ \sum_{\substack{0 \leq i,j,k < m \\ i+j+k=m+2}} n_{ijk} \langle \mth{T}{i}, \mth{T}{j} \rangle \mth{T}{k},
  \end{align}
  where $n_{ijk} = n_{ijk}(m)$ are non-negative integers.
  Indeed,
  \begin{align*}
    Z_3 = &~ (Z_2)_s + \abs{\mth{T}{1}}^2 \mth{T}{3} \\
    = &~ 2 \langle \mth{T}{4}, \mth{T}{1} \rangle T \\
    \nonumber &+ 6 \langle \mth{T}{3}, \mth{T}{2} \rangle T \\
    \nonumber &+ 8 \langle \mth{T}{3}, \mth{T}{1} \rangle \mth{T}{1} \\
    \nonumber &+ 4 \abs{\mth{T}{1}}^2 \mth{T}{3} \\
    \nonumber &+ 8 \abs{\mth{T}{2}}^2 \mth{T}{1} \\
    \nonumber &+ 12 \langle \mth{T}{2}, \mth{T}{1} \rangle \mth{T}{2}.
  \end{align*}
  Now assume that \cref{eq:mthformula} holds for some $m$. Then we have that
  \begin{align*}
    Z_{m+1} = &~ (Z_m)_s + \abs{T_s}^2 \mth{T}{m+1} \\
    = &~ 2 \langle \mth{T}{m+2}, \mth{T}{1} \rangle T + 2 \langle \mth{T}{m+1}, \mth{T}{2} \rangle T + 2 \langle \mth{T}{m+1}, \mth{T}{1} \rangle \mth{T}{1} \\
    \nonumber &+ 2m \langle \mth{T}{m+1}, \mth{T}{2} \rangle T + 2m \langle \mth{T}{m}, \mth{T}{3} \rangle T + 2m \langle \mth{T}{m}, \mth{T}{2} \rangle \mth{T}{1} \\
    \nonumber &+ 2(m+1) \langle \mth{T}{m+1}, \mth{T}{1} \rangle \mth{T}{1} + 2(m+1) \langle \mth{T}{m}, \mth{T}{2} \rangle \mth{T}{1} \\
    \nonumber &+ 2(m+1) \langle \mth{T}{m}, \mth{T}{1} \rangle \mth{T}{2} + 2(m+1) \langle \mth{T}{2}, \mth{T}{1} \rangle \mth{T}{m} \\
    \nonumber &+ (m+1) \abs{\mth{T}{1}}^2 \mth{T}{m+1} \\
    \nonumber &+ \sum_{\substack{0 \leq i,j,k < m \\ i+j+k=m+2}} n_{ijk} \langle \mth{T}{i+1}, \mth{T}{j} \rangle \mth{T}{k} \\
    \nonumber &+ \sum_{\substack{0 \leq i,j,k < m \\ i+j+k=m+2}} n_{ijk} \langle \mth{T}{i}, \mth{T}{j+1} \rangle \mth{T}{k} \\
    \nonumber &+ \sum_{\substack{0 \leq i,j,k < m \\ i+j+k=m+2}} n_{ijk} \langle \mth{T}{i}, \mth{T}{j} \rangle \mth{T}{k+1} + \abs{T_s}^2 \mth{T}{m+1} \\
    = &~ 2 \langle \mth{T}{m+2}, \mth{T}{1} \rangle T + 2(m+1) \langle \mth{T}{m+1}, \mth{T}{2} \rangle T \\
    \nonumber &+ 2(m+2) \langle \mth{T}{m+1}, \mth{T}{1} \rangle \mth{T}{1} + (m+2) \abs{\mth{T}{1}}^2 \mth{T}{m+1} \\
    \nonumber &+ \sum_{\substack{0 \leq i,j,k < m+1 \\ i+j+k=m+3}} \tilde{n}_{ijk} \langle \mth{T}{i}, \mth{T}{j} \rangle \mth{T}{k},
  \end{align*}
  proving the claim.

  We obtain
  \begin{align*}
    \abs{\mth{T}{m}}^2_t - \abs{\mth{T}{m}}^2_{ss} = &~ 2 \langle Z_m, \mth{T}{m} \rangle - 2 \abs{\mth{T}{m+1}}^2 \\
    = &~ 4 \langle \mth{T}{m+1}, \mth{T}{1} \rangle \langle T, \mth{T}{m} \rangle + 4m \langle \mth{T}{m}, \mth{T}{2} \rangle \langle T, \mth{T}{m} \rangle \\
      \nonumber &+ 4(m+1) \langle \mth{T}{m}, \mth{T}{1} \rangle^2 + 2(m+1) \abs{\mth{T}{1}}^2 \abs{\mth{T}{m}}^2 \\
      \nonumber &+ 2 \sum_{\substack{0 \leq i,j,k < m \\ i+j+k=m+2}} n_{ijk} \langle \mth{T}{i}, \mth{T}{j} \rangle \langle \mth{T}{k}, \mth{T}{m} \rangle \\
      \nonumber &- 2 \abs{\mth{T}{m+1}}^2 \\
    \leq &~ - 2 \abs{\mth{T}{m+1}}^2 + 4 \abs{\mth{T}{m+1}} \abs{\mth{T}{1}} \abs{\mth{T}{m}} \\
      \nonumber &- 2 \abs{\mth{T}{1}}^2 \abs{\mth{T}{m}}^2 + 2(m+2) \abs{\mth{T}{1}}^2 \abs{\mth{T}{m}}^2 \\
      \nonumber &+ 4m \langle \mth{T}{m}, \mth{T}{2} \rangle \langle T, \mth{T}{m} \rangle + 4(m+1) \langle \mth{T}{m}, \mth{T}{1} \rangle^2 \\
      \nonumber &+ 2 \sum_{\substack{0 \leq i,j,k < m \\ i+j+k=m+2}} n_{ijk} \langle \mth{T}{i}, \mth{T}{j} \rangle \langle \mth{T}{k}, \mth{T}{m} \rangle \\
    = &~ - 2 \left( \abs{\mth{T}{m+1}} - \abs{\mth{T}{1}} \abs{\mth{T}{m}} \right)^2 \\
      \nonumber &+ 2(m+2) \abs{\mth{T}{1}}^2 \abs{\mth{T}{m}}^2 \\
      \nonumber &+ 4m \langle \mth{T}{m}, \mth{T}{2} \rangle \langle T, \mth{T}{m} \rangle + 4(m+1) \langle \mth{T}{m}, \mth{T}{1} \rangle^2 \\
      \nonumber &+ 2 \sum_{\substack{0 \leq i,j,k < m \\ i+j+k=m+2}} n_{ijk} \langle \mth{T}{i}, \mth{T}{j} \rangle \langle \mth{T}{k}, \mth{T}{m} \rangle.
  \end{align*}
  By the induction hypothesis, we have
  \begin{equation*}
    \abs{\mth{T}{i}}^2 \leq \frac{C_i K_0}{t^{i-1}}
  \end{equation*}
  for any $i = 1, \dots, m-1$. Thus
  \begin{align*}
    \abs{\mth{T}{m}}^2_t - \abs{\mth{T}{m}}^2_{ss} \leq &~ A_1 K_0 \abs{\mth{T}{m}}^2 + A_2 \sqrt{\frac{K_0}{t}} \abs{\mth{T}{m}}^2 \\
    \nonumber &+ 2 \sum_{\substack{0 \leq i,j,k < m \\ i+j+k=m+2}} n_{ijk} \abs{\mth{T}{i}} \abs{\mth{T}{j}} \abs{\mth{T}{k}} \abs{\mth{T}{m}} \\
    \leq &~ A_3 K_0 \abs{\mth{T}{m}}^2 + A_4 K_0^3 t^{-(m-2)} + \frac{A_5}{t} \abs{\mth{T}{m}}^2,
  \end{align*}
  using the Peter--Paul inequality with $\epsilon=t$, where the constants $A_i$ depend on $m$ and $C_1, \dots, C_{m-1}$.
  Therefore,
  \begin{align*}
    \left( t^{m-1} \abs{\mth{T}{m}}^2 \right)_t - \left( t^{m-1} \abs{\mth{T}{m}}^2 \right)_{ss} \leq &~ (m-1) t^{m-2} \abs{\mth{T}{m}}^2 \\
    \nonumber &+ A_3 K_0 t^{m-1} \abs{\mth{T}{m}}^2 \\
    \nonumber &+ A_4 K_0^3 t + A_5 t^{m-2} \abs{\mth{T}{m}}^2.
  \end{align*}
  For a large enough constant $C>0$, we set $\Phi_m = t^{m-1} \abs{\mth{T}{m}}^2 + C t^{m-2} \abs{\mth{T}{m-1}}^2$ and obtain
  \begin{align*}
    (\Phi_m)_t - (\Phi_m)_{ss} &\leq t^{m-2} (A_3 K_0 t + A_6 - 2 C) \abs{\mth{T}{m}}^2 + A_7 K_0^2 \\
    &\leq A_7 K_0^2.
  \end{align*}
  We then proceed as in the $m=2$ case to obtain $C_m$.
\end{proof}

Using these estimates and the short-time existence, we have long-time existence in the sense that as long as the curvature stays bounded, the flow can be continued for some time. In particular, the torsions do not play a role.
\begin{theorem}[cf. {\cite[Thm.~1.13]{altschuler1992shortening}}]
  \label{thm:LongTimeExistence}
  Assume that the curvature $\kappa$ is bounded on the time interval $[0, t_0)$. Then there exists $\epsilon > 0$ such that the curve shortening flow $\{\gamma_t\}$ exists and is smooth on the interval $[0, t_0+\epsilon)$.
\end{theorem}

Equivalently, we may say that if $\omega $ is the maximal time of existence of the flow, the curvature must tend to infinity as $t$ approaches $\omega $.
\begin{corollary}
  Suppose that $\gamma: S^1 \times [0, \omega ) \to \R^n$ is a solution of \labelcref{eq:CSF} with initial data $\gamma(\cdot, 0) = \gamma_0$. Then $\omega $ is finite, and furthermore, $\max_{\gamma_t} \kappa^2 \to \infty$ as $t \to \omega $.
\end{corollary}

\section{Blow-up limits}
\label{sec:BlowUpLimits}

Recall that we have set $K_t = \sup \kappa^2(\cdot, t)$.
Assume that $\{\gamma_t\}$ is a curve shortening flow with a singularity forming at time $\omega $. If there exists a constant $c>0$ such that
\begin{equation*}
  K_t \leq \frac{c}{\omega -t}, \qquad t < \omega ,
\end{equation*}
we say that the singularity at $\omega $ is of type-I \cite{altschuler1991singularities}. Otherwise, that is, if
\begin{equation*}
  \limsup_{t \to T} K_t (\omega -t) = \infty,
\end{equation*}
we say that it is of type-II.

Moreover, we say that $\{(p_j,t_j)\} \subset S^1 \times [0,\omega )$ is a blow-up sequence if $t_j \to \omega $ as $j \to \infty$ and
\begin{equation*}
    \lim_{j \to \infty} \kappa^2(p_j, t_j) = \infty.
\end{equation*}
In particular, a blow-up sequence is called essential if there exists a constant $\rho > 0$ such that
\begin{equation*}
  \rho K_t \leq \kappa^2(p_j, t_j), \qquad t < t_j.
\end{equation*}

For any immersed curve $\gamma: S^1 \to \R^n$, the total absolute curvature $\int_\gamma \abs{\kappa} \dx{s}$ is a scaling-invariant quantity.
In order to show that blow-up limits of the curve shortening flow are planar, we recall Altschuler's estimate on the derivative of the total absolute curvature, which also holds for evolving curves in $\R^n$ due to the evolution equations derived in \cref{subsec:EvolEqns}.

\begin{theorem}[cf. {\cite[Thm.~5.1]{altschuler1991singularities}}, see also {\cite{yang2005curve}}]
  \label{thm:TACestimate}
  Let $\gamma: S^1 \times [0, \omega ) \to \R^n$ be a solution of \labelcref{eq:CSF}. Then the integral estimate
  \begin{equation*}
    \td{}{t} \int_\gamma \abs{\kappa} \dx{s} \leq -\int_\gamma \abs{\kappa} \tau_1^2 \dx{s}
  \end{equation*}
  holds for $t \in [0, \omega )$.
\end{theorem}
\begin{proof}
  From \cref{cor:EvolKSquared}, we have that
  \begin{equation*}
  	\pd{\kappa^2}{t} = \pdd{\kappa^2}{s}  - 2 \left( \pd{\kappa}{s} \right)^2 + 2\kappa^4 -  2 \kappa^2 \tau_1^2.
  \end{equation*}
  As in Altschuler's proof, we then define $\kappa_\epsilon = \sqrt{\kappa^2 + \epsilon}$, where $\epsilon > 0$ is arbitrary, and obtain
  \begin{equation*}
    \td{}{t} \int_\gamma \kappa_\epsilon \dx{s} \leq - \int_\gamma \frac{1}{\kappa_\epsilon} \kappa^2 \tau_1^2 \dx{s},
  \end{equation*}
  which implies the claim.
\end{proof}

For a planar curve we obtain a more precise formula.
\begin{theorem}[Altschuler {\cite[Thm.~5.14]{altschuler1991singularities}}]
  \label{thm:TACPlanar}
  For a planar solution $\gamma$ to the curve shortening flow, we have
  \begin{equation*}
    \td{}{t} \int_\gamma \abs{\kappa} \dx{s} = - 2 \sum_{\{ p \;:\; \kappa(p, \cdot) = 0 \}} \abs{\pd{\kappa}{s}}.
  \end{equation*}
\end{theorem}

Given a blow-up sequence $\{(p_j,t_j)\}$, we define a blow-up procedure as follows \cite{altschuler1991singularities}: Define $\gamma_j: S^1 \times [\alpha_j, \omega_j) \to \R^n$, where $\alpha_j = -\lambda_j^2 t_j$, $\omega_j = \lambda_j^2 (\omega - t_j)$, by
\begin{equation*}
  \gamma_j(\cdot, \bar{t}) = \lambda_j (A_j \gamma(\cdot, t) + b_j), \qquad \bar{t} = \lambda_j^2(t-t_j).
\end{equation*}
Moreover, $\lambda_j > 0$, $A_j \in \SO(n)$, $b_j \in \R^n$ are such that
\begin{align*}
  \gamma_j(p_j, 0) &= 0 \in \R^n, \\
  T_j(p_j, 0) &= (1, 0, \dots, 0) = e_1, \\
  N_j(p_j, 0) &= (0, 1, 0, \dots, 0) = e_2, \\
  (B_i)_j(p_j, 0) &= (0,\dots,0,1,0,\dots,0) = e_{i+2}, \qquad i=1,\dots,n-2.
\end{align*}

In order to parametrise $\gamma_j$ by arclength, note that
\begin{equation*}
  \abs{\pd{\gamma_j}{s}} = \lambda_j \abs{A_j \pd{\gamma}{s}} = \lambda_j.
\end{equation*}
Therefore, if we let $\bar{s} = \lambda_j s$ then $\gamma_j$ is parametrised by arclength $\bar{s}$ once we define
\begin{equation*}
  \gamma_j(\bar{s}, \bar{t}) = \lambda_j (A_j \gamma(s, t) + b_j).
\end{equation*}
This then gives
\begin{equation*}
  \pd{\gamma_j}{\bar{s}} = \pd{\gamma_j}{s} \pd{s}{\bar{s}} = \lambda_j A_j \pd{\gamma}{s} \frac{1}{\lambda_j} = A_j \pd{\gamma}{s},
\end{equation*}
that is,
\begin{equation}\label{eq:Tj}
  T_j(\bar{s}, \bar{t}) = A_j T(s, t),
\end{equation}
and in the same way we obtain
\begin{align*}
  N_j(\bar{s}, \bar{t}) &= A_j N(s, t), \\
  (B_i)_j(\bar{s}, \bar{t}) &= A_j B_i(s,t), \\
  \kappa_j(\bar{s}, \bar{t}) &= \frac{1}{\lambda_j} \kappa(s, t), \\
  (\tau_i)_j(\bar{s},\bar{t}) &= \frac{1}{\lambda_j} \tau_i(s, t).
\end{align*}
Then $\gamma_j$ is indeed parametrised by arclength $\bar{s}$,
\begin{equation*}
  \abs{\pd{\gamma_j}{\bar{s}}} = \abs{A_j \pd{\gamma}{s}} = 1,
\end{equation*}
since $\gamma$ is parametrised by arclength $s$ and $\abs{A_j} = 1$.
As a result, each $\gamma_j$ is a solution of curve shortening flow,
\begin{equation*}
  \pd{\gamma_j}{\bar{t}}(\bar{s}, \bar{t}) = (\kappa_j N_j)(\bar{s}, \bar{t}).
\end{equation*}
Indeed,
\begin{align*}
  \pd{\gamma_j}{\bar{t}}(\bar{s}, \bar{t}) &= \lambda_j A_j \pd{\gamma}{t} (s, t) \pd{t}{\bar{t}} \\
  &= \frac{1}{\lambda_j} A_j (\kappa N)(s, t) \\
  &= (\kappa_j N_j)(\bar{s}, \bar{t}).
\end{align*}

For the following theorem, we can use the exact same proof as Altschuler \cite{altschuler1991singularities}, which remains valid in view of the evolution equations of \cref{subsec:EvolEqns} and the scaling invariant estimates of \cref{thm:CurvatureEstimates}. We repeat it here for convenience.

\begin{theorem}[cf. {\cite[Thm.~7.3]{altschuler1991singularities}}]
  \label{thm:BlowUpLimit}
  Let $\gamma: S^1 \times [0, \omega) \to \R^n$ be a solution of \labelcref{eq:CSF}.
  Assume that $\{(p_j, t_j)\}$ is an essential blow-up sequence. Then there exists a subsequence of $\{(p_j, t_j)\}$ along which the rescaled solutions $\gamma_j$ converge to a smooth nontrivial limit $\gamma_\infty$ which exists at least on the time interval $[-\infty, 0]$.
\end{theorem}
\begin{proof}
  Set $\lambda_j := \kappa(p_j, t_j)$, so that $\kappa_j^2(p_j, 0) = 1$ (note that $\bar{t} = 0 \Leftrightarrow t = t_j$).

  Since $\{(p_j, t_j)\}$ is a blow-up sequence, we have that $\lim_{j \to \infty} \alpha_j = - \infty$.
  If a type-I singularity occurs (that is, $\lim_{t \to \omega} K_t (\omega - t) < \infty$), then $\lim_{j \to \infty} \omega_j < \infty$, for a type-II singularity ($\lim_{t \to \omega} K_t (\omega - t) = \infty$), we can choose an essential blow-up sequence such that $\lim_{j \to \infty} \omega_j = \infty$, since from the definition of $\lambda_j$ we have that $\lambda_j^2 \leq K_{t_j}$.

  As a limit solution might be a family of noncompact curves, we consider the solutions $\gamma_j$ instead as a family of curves $\tilde{\gamma}_j: \R \times [\alpha_j, \omega_j) \to \R^n$, periodic in space, such that $\tilde{\gamma}_j(0, \cdot) = \gamma_j(p_j, \cdot)$.
  Denote the arclength parameter for $\tilde{\gamma}_j(\cdot, \bar{t})$ from the origin $0 \in \R$ by $\bar{s}$ and recall that, in general, $\bar{s}$ depends on $\bar{t}$ in the sense that $\pd{\bar{s}}{\bar{t}} \neq 0$.

  Define the differential operator
  \begin{equation*}
    \dd{}{\bar{t}} = \pd{}{\bar{t}} + \phi_j(\bar{s}) \pd{}{\bar{s}},
  \end{equation*}
  where $\phi_j(\bar{s}) = \int_0^{\bar{s}} \kappa_j^2(\sigma,\bar{t}) \dx{\sigma}$ and thus $\pd{\phi_j}{\bar{s}} = \kappa_j^2$. Then,
  \begin{align*}
    \left[ \dd{}{\bar{t}}, \pd{}{\bar{s}} \right] &= \pd{}{\bar{t}} \pd{}{\bar{s}} + \phi_j \pd{}{\bar{s}} \pd{}{\bar{s}} - \pd{}{\bar{s}} \pd{}{\bar{t}} - \pd{\phi_j}{\bar{s}} \pd{}{\bar{s}} - \phi_j \pd{}{\bar{s}} \pd{}{\bar{s}} \\
    &= \pd{}{\bar{t}} \pd{}{\bar{s}} - \pd{}{\bar{s}} \pd{}{\bar{t}} - \kappa_j^2 \pd{}{\bar{s}} \\
    &= 0.
  \end{align*}
  Therefore, denoting $v = \abs{\pd{\gamma}{p}}$ and $\bar{s} = \int_{p_0}^p v \dx{q}$,
  \begin{align*}
    \dd{\bar{s}}{\bar{t}} &= \pd{\bar{s}}{\bar{t}} + \phi_j \\
    &= - \int_{p_0}^p \kappa_j^2 v \dx{q} + \int_0^{\bar{s}} \kappa_j^2 \dx{\sigma} \\
    &= 0,
  \end{align*}
  where we have used that $\pd{v}{\bar{t}} = - \kappa_j^2 v$ and $\dx{\bar{s}} = v \dx{p}$.

  Since $\{(p_j, t_j)\}$ is an essential blow-up sequence, there exists $\rho > 0$ independent of $j$ such that $\rho K_t \leq \kappa^2(p_j, t_j)$ whenever $t \leq t_j$. In particular, for the curves $\tilde{\gamma}_j$, $\rho \sup \kappa_j^2(\cdot, \bar{t}) \leq \kappa_j^2(p_j, 0) = 1$ for $\bar{t} \leq 0$.

  Then differentiating \cref{eq:Tj} with respect to $\bar{s}$ we get
  \begin{equation*}
    \pd{T_j}{\bar{s}} = A_j \pd{T}{s} \pd{s}{\bar{s}} = \frac{1}{\lambda_j} A_j \pd{T}{s}.
  \end{equation*}
  Hence, by \cref{thm:CurvatureEstimates}, we have that
  \begin{equation*}
    \abs{\pd{T_j}{\bar{s}}}^2 = \frac{1}{\lambda_j^2} \abs{\pd{T}{s}}^2 \leq \frac{\tilde{c}_1 K_{t_j}}{\lambda_j^2} \leq c_1,
  \end{equation*}
  since $\{(p_j, t_j)\}$ is an essential blow-up sequence and we have $\frac{K_{t_j}}{\lambda_j^2} \leq \frac{1}{\rho} < \infty$.

  Taking further derivatives with respect to $\bar{s}$, we obtain
  \begin{equation*}
    \pdn{T_j}{\bar{s}}{\ell} = \frac{1}{\lambda_j^\ell} A_j \pdn{T}{s}{\ell},
  \end{equation*}
  whereby
  \begin{equation*}
    \abs{\pdn{T_j}{\bar{s}}{\ell}}^2 = \frac{1}{\lambda_j^{2 \ell}} \abs{\pdn{T}{s}{\ell}}^2 \leq c_\ell.
  \end{equation*}

  Using the commutation formula for $\pd{}{\bar{t}}$ and $\pd{}{\bar{s}}$, we get
  \begin{equation*}
    \abs{\pd{T_j}{\bar{t}}}^2 \leq \abs{\pdd{T_j}{\bar{s}}}^2 + \abs{\pd{T_j}{\bar{s}}}^4 \leq c_2 + c_1^2.
  \end{equation*}
  Then, taking derivatives with respect to $\bar{t}$ and using the commutation formula, we see that $\abs{\pdn{T_j}{\bar{t}}{\ell}}^2$ is bounded by a sum of products of $\abs{\pdn{T_j}{\bar{t}}{k}}^2$ for $k < \ell$ and $\abs{\pdn{T_j}{\bar{s}}{m}}^2$ for $m \leq 2 \ell$ and is thus itself bounded.

  By the definition of $\dd{}{\bar{t}}$, therefore, the fact that $\dd{}{\bar{t}}$ and $\pd{}{\bar{s}}$ commute, and the estimate $\phi_j(\bar{s}) \leq \rho^{-1} \bar{s}$ we conclude
  that $\abs{\ddn{T_j}{\bar{t}}{\ell}}^2$ is bounded for any $\ell$ independently of $j$ on compact subsets of $\R \times [-\infty, \omega_\infty)$, %
  and, again, since $\dd{}{\bar{t}}$ and $\pd{}{\bar{s}}$ commute, the same is true for all mixed derivatives $\abs{\frac{\delta^j \partial^k T_j}{\delta^j \bar{t} \; \partial^k \bar{s}}}^2$. %

  By the Arzelà--Ascoli theorem, there exists a subsequence of $\{(p_j, t_j)\}$, denoted the same, along which the tangent vectors $T_j(\bar{s}, \bar{t})$ converge uniformly on compact sets of $\R \times [-\infty, \omega_\infty)$ to a smooth limit $T_\infty(\bar{s}, \bar{t})$ as $j \to \infty$. We may thus define a smooth limit solution $\tilde{\gamma}_\infty$ by integrating $T_\infty$.
  If $\tilde{\gamma}_\infty$ is periodic, we denote by $\gamma_\infty$ one period of $\tilde{\gamma}_\infty$, if not, we set $\gamma_\infty = \tilde{\gamma}_\infty$.

  Finally, $\gamma_\infty$ cannot be trivial, i.\,e., a straight line, for
  \begin{equation*}
    \kappa_\infty^2(0, 0) = \lim_{j \to \infty} \tilde{\kappa}_j^2(0, 0) = \lim_{j \to \infty} \kappa_j^2(p_j, 0) = 1,
  \end{equation*}
  finishing the proof.
\end{proof}

The fact that blow-up limits of evolving space curves are planar goes back to Altschuler's work \cite{altschuler1991singularities}. For curves in any codimension this fact was pointed out by Yang and Jiao \cite{yang2005curve}. Using \cref{thm:TACestimate} and the monotonicity formula, we can give a simple proof also in the general case.%

\begin{theorem}[cf. {\cites[Thm.~7.7]{altschuler1991singularities}}]
  \label{thm:BlowUpIsPlanar}
  Let $\gamma: S^1 \times [0, \omega ) \to \R^n$ be a solution of \labelcref{eq:CSF}. Then any nontrivial blow-up limit of $\gamma$ is planar and convex.
\end{theorem}

\begin{proof}
  By Huisken's monotonicity formula, any blow-up limit $\gamma_\infty$ of curve shortening flow is self-similar. Moreover, the total absolute curvature is scaling-invariant. Therefore, \cref{thm:TACestimate} implies that
  \begin{equation*}
    0 \leq -\int_{\gamma_\infty} \abs{\kappa} \tau_1^2 \dx{s}\!.
  \end{equation*}
  This implies that, at almost every point of the smooth limit curve $\gamma_\infty$, we must have either $\kappa=0$ or $\tau_1=0$. Then \cref{thm:SpivakThm} implies that $\gamma_\infty$ must be contained in a $2$-dimensional subspace of $\R^n$.
  Note that $\gamma_\infty$ cannot have any inflection points, since by \cref{thm:TACPlanar}, any inflection point must be degenerate, that is, $\kappa = \pd{\kappa}{s} = 0$, but a result of Angenent \cite{angenent1991parabolic} implies that any solution with degenerate inflection points must be a line.
\end{proof}

\subsection{Type-I singularities}

In order to analyse the behaviour of type-I singularities, we can employ Huisken's argument for singularities of mean curvature flow \cite{huisken1990asymptotic}, which was also used by Altschuler \cite{altschuler1991singularities} to prove the corresponding result for curve shortening flow of space curves.

To that end, let $\gamma: S^1 \times [0, \omega ) \to \R^n$ be a solution of \labelcref{eq:CSF} and assume that $(0,\omega ) \in \R^n \times \R$ is a singular point of type-I reached by the flow. We then define a continuous rescaling of the flow via
\begin{equation*}
  \tilde{\gamma}(s, \tilde{t}) = \frac{1}{\sqrt{2(\omega - t)}} \gamma(s,t),
\end{equation*}
where $\tilde{t} = - \frac{1}{2} \log(\omega -t)$.
The rescaled flow $\{\tilde \gamma_t\}$ is thus defined for $- \frac{1}{2} \log \omega  \leq \tilde t < \infty$, and
in terms of the differential operators
\begin{align*}
  \pd{}{\tilde t} &= 2(\omega -t) \pd{}{t}, \\
  \pd{}{\tilde s} &= \sqrt{2(\omega -t)} \pd{}{s},
\end{align*}
it satisfies
\begin{equation*}
  \pd{}{\tilde t} \tilde \gamma = \pdd{}{\tilde s} \tilde \gamma + \tilde \gamma.
\end{equation*}
The reason to choose this particular rescaling is that the type-I assumption then implies that the curvature of the rescaled flow is uniformly bounded for all time, since
\begin{equation*}
  \tilde \kappa(s, \tilde t) = \sqrt{2(\omega -t)} \kappa(s, t).
\end{equation*}

In the rescaled setting, we then have the monotonicity formula \cite{huisken1990asymptotic} %
\begin{equation}\label{eq:RescaledMonotonicityFormula}
  \td{}{\tilde t} \int_{\tilde \gamma} \tilde k \dx{\tilde s} = - \int_{\tilde \gamma} \abs{\pd{\tilde \gamma}{\tilde s} + \tilde \gamma^\perp}^2 \tilde k \dx{\tilde s},
\end{equation}
where the rescaled backwards heat kernel on $\R^n$ is given by
\begin{equation*}
  \tilde k(x, \tilde t) = \e^{- \abs{x}^2}.
\end{equation*}

We can then perform the blow-up procedure as in the proof of \cref{thm:BlowUpLimit}, noting that the type-I assumption also implies that any blow-up sequence is necessarily essential, to obtain a subsequential limit. By the rescaled monotonicity formula \labelcref{eq:RescaledMonotonicityFormula}, we conclude that the limit is self-similarly shrinking \cite{huisken1990asymptotic}, and by \cref{thm:BlowUpIsPlanar} the limit is planar. Moreover, by continuity of the total absolute curvature, the winding number cannot change. We thus have

\begin{theorem}[cf. {\cite[Thm.~8.15]{altschuler1991singularities}}]
  \label{thm:ConvergenceTypeI}
  Suppose that a type-I singularity is forming at time $\omega $.
  Let $\{(p_j,t_j)\}$ be a blow-up sequence. Then there exists a subsequence of $\{(p_j,t_j)\}$ such that a rescaling of $\gamma$ along it converges to a planar self-similarly shrinking solution $\gamma_\infty$ with the same winding number.
\end{theorem}

\subsection{Type-II singularities}

We now assume that the singular point $(0,\omega ) \in \R^n \times \R$ reached by the flow is of type-II.
Since the argument is exactly the same as for space curves, we do not repeat the details and instead refer to Altschuler's work \cite{altschuler1991singularities}.

We already know that by \cref{thm:BlowUpLimit}, a limit of rescalings $\gamma_\infty$ must exist on the interval $[-\infty, 0]$. Moreover, it is planar and convex.
It is then possible to show that, since the singularity is of type-II, there exists an essential blow-up sequence such that a limit of rescalings along it is in fact eternal, that is, it exists on the time interval $[-\infty, \infty]$. In addition, the limit solution is embedded and its total curvature is equal to $\pi$. Furthermore, by showing that the curvature and all its derivatives tend to zero at the ends, one then proves that this limit must be the Grim Reaper. Finally, one has

\begin{theorem}[cf. {\cite[Thm.~8.16]{altschuler1991singularities}}]
  \label{thm:ConvergenceTypeII}
  Suppose that a type-II singularity is forming at time $\omega $.
  Then there exists an essential blow-up sequence $\{(p_j,t_j)\}$ such that a sequence of rescalings along it converges to the Grim Reaper.
\end{theorem}

\section{Convergence analysis}
\label{sec:ConvergenceAnalysis}

We now come to the proof of our main theorem, that is, we show that initial curves with entropy less than that of the Grim Reaper converge to a round point in finite time.

\begin{proof}[Proof of \cref{thm:CSFConvergence}]
  Since the entropy is non-increasing under curve shortening flow, we may assume that $\lambda(\gamma) < 2$. For if $\lambda(\gamma)$ was equal to $2$, the initial curve $\gamma_0$ would have to be a self-shrinker, but $\lambda(S^1) < 2$.

  Assume that a type-II singularity forms at time $\omega $. Then by \cref{thm:ConvergenceTypeII}, there exists an essential blow-up sequence such that a limit $\gamma_\infty$ of rescalings along it is the Grim Reaper. Since the entropy is lower semicontinuous with regard to the locally smooth convergence, we must have $\lambda(\gamma_\infty) < 2$, however, we know from \cref{prop:EntropyGR} that the entropy of the Grim Reaper equals $2$. Thus the singularity cannot be of type-II.

  Therefore, assume that a type-I singularity forms, so that by \cref{thm:ConvergenceTypeI}, we have a limit $\gamma_\infty$ of a sequence of rescalings $\{\gamma_j\}$ that is self-similarly shrinking in the plane. By the classification of Abresch--Langer \cite{abresch1986normalized}, $\gamma_\infty$ could be one or more lines through the origin, a singly or multiply-covered circle, or one of the Abresch--Langer curves $\gamma_{m,n}$, $m \geq 2$.
  \Cref{prop:EntropyAL} implies that $\lambda(\gamma_{m,n}) \geq m \sqrt{\frac{2 \pi}{\e}}$, so that the latter possibility cannot occur.
  Moreover, from \cref{prop:EntropyEuclideanDensity} we have that $\lambda(\gamma_\infty) \geq \Theta^2(\gamma_\infty, x)$ for any point $x$, which implies that $\gamma_\infty$ is embedded. Hence $\gamma_\infty$ cannot be a multiply-covered circle or a family of intersecting lines.

  By standard theory \cite{brakke1978motion,white2005local}, should a single line appear as a limit of rescalings of the flow, the fact that its Gaussian density is $1$ in a suitable space-time region implies that the curvature is bounded there after all, so that the blow-up point is not a singularity, which is a contradiction. Therefore, the blow-up limit $\gamma_\infty$ must be the standard circle around the origin, that is, the tangent flow at the singularity is a smooth, closed, embedded self-shrinker. Then Schulze's uniqueness result for compact tangent flows \cite{schulze2014uniqueness} implies that this is the only possible tangent flow, whereby the rescaled flow converges to the round circle.
\end{proof}

\printbibliography

\end{document}